\documentclass[12pt,leqno]{article}
\usepackage{amsmath,amssymb,amsthm,graphicx,float}

\newcommand{\bbC}{\mathbb{C}}

\newcommand{\bbH}{\mathbb{H}}
\newcommand{\bbN}{\mathbb{N}}
\newcommand{\bbR}{\mathbb{R}}

\newcommand{\gve}{\varepsilon}

\newcommand{\ga}{\alpha}
\newcommand{\gb}{\beta}
\newcommand{\gd}{\delta}
\newcommand{\gk}{\kappa}
\newcommand{\gl}{\lambda}
\newcommand{\glt}{\widetilde{\lambda}}
\newcommand{\go}{\omega}
\newcommand{\gr}{\gamma}

\newcommand{\gs}{\sigma}

\newcommand{\gD}{\Delta}

\newcommand{\gR}{\Gamma}

\newcommand{\gO}{\Omega}
\renewcommand{\Re}{\text{Re}\,}
\renewcommand{\Im}{\text{Im}\,}

\newcommand{\diam}{\text{diam}}

\newcommand{\half}{\frac{1}{2}}

\newcommand{\Lip}{\text{Lip}}
\newcommand{\LipHalfNorm}{\mbox{$\Vert \gl \Vert_{\text{Lip}(\frac{1}{2})}$}}
\newcommand{\LipHalfNormt}{\mbox{$\Vert \widetilde{\gl} \Vert_{\text{Lip}(\frac{1}{2})}$}}
\newcommand{\MTdeltaCondition}{$(M,T,\delta)$-Lip($\frac{1}{2}+\delta$) condition}
\newcommand{\sigmaCondition}{$\sigma$-Lip($\frac{1}{2}$) condition}
\newcommand{\MTnAlphaCondition}{$(M,T,n,\alpha)$-$C^{n,\alpha}$ condition}
\newcommand{\MToneDeltaCondition}{$(M,T,1,\delta)$-$C^{1,\delta}$ condition}

\newtheorem{theorem}{Theorem}[section]
\newtheorem{lemma}[theorem]{Lemma}
\newtheorem{corollary}[theorem]{Corollary}
\newtheorem{proposition}[theorem]{Proposition}
\newtheorem*{main}{Main Result}

\newenvironment*{refThm}[1]{ 
  \def\thethmid{#1} 
  \begin{the-refThm} 
}{ 
  \end{the-refThm} 
} 
\newtheorem*{the-refThm}{Theorem \thethmid} 

\theoremstyle{definition}
\newtheorem*{definition}{Definition}
\newtheorem*{example}{Example}

\newtheorem*{fact}{Fact}
\newtheorem*{notation}{Notation}
\newtheorem*{gnotation}{General notation/convention}
\newtheorem*{acknowledgement}{Acknowledgement}

\newenvironment{thmList}{
  \begin{enumerate}

}{\end{enumerate}}

\newenvironment{defList}{
  \begin{enumerate}

}{\end{enumerate}}

\begin{document}

\title{Smoothness of Loewner Slits}
\author{Carto Wong}
\date{\today}
\maketitle

\begin{abstract}
  In this paper, we show that the chordal Loewner differential equation with 
  $C^{\gb}$ driving function generates a $C^{\gb+\half}$ slit for $\half < \gb \leq 2$, except
  when $\gb = \frac{3}{2}$ the slit is only proved to be weakly $C^{1,1}$.
\end{abstract}

\section{Introduction}

% ===============================
%   NEED to update: page 3 [ZT] -> [RTZ]
% ===============================   

The Loewner differential equation is a classical tool in complex analysis
which has been successfully applied to various extremal problems,
including the famous de Branges theorem (see \cite{Bra}). In recent years, the
Schramm-Loewner evolution (SLE) has been extensively studied by
mathematicians and physicists. One can think of SLE as a random curve in the
upper half-plane, which is generated via Loewner differential equation with a
random driving function. In the meanwhile, some
natural questions in the deterministic side of SLE are still open. In this paper,
we investigate the smoothness of slits generated by $C^{\gb}$ driving functions.

Given a slit (definition in \S~\ref{S:definition}) $\gr \colon [0,T] \to \overline{\bbH}$ in
the upper half-plane $\bbH = \{ z \in \bbC \colon \Im z > 0 \}$, the region $H_t := \bbH
\setminus \gr([0,t])$ is simply connected for each $t$. There is a unique conformal
map $g_t \colon H_t \to \bbH$ satisfying the hydrodynamic normalization
\[
    g_t(z) = z + \frac{a_1(t)}{z} + \frac{a_2(t)}{z^2} + \cdots
\]
as $z \to \infty$. Figure~\ref{Fig:chordalLE} illustrates the situation. The coefficient
$a_1(t)$ is called the \emph{half-plane capacity} of the set $K_t = \gr([0,t])$. See \cite{RW}
for a geometric interpretation of half-plane capacity and its relation to conformal radius, and
see \cite{LLN} for a probabilistic approach. Although it is not immediate from the definition,
it is routine to show that $a_1(t)$ is a strictly increasing real-valued function with $a_1(0) = 0$.
If the slit is parametrized so that $a_1(t) = 2t$, then $g_t(z)$ is differentiable in $t$ and satisfies
the \emph{chordal Loewner differential equation}
\begin{equation} \label{LDE:forward_chordal}
    \left\{ \begin{aligned}       \frac{\partial}{\partial t} g(t,z) &= \frac{2}{g(t,z) - \gl(t)} \\
      g(0,z) &= z
    \end{aligned} \right.  
\end{equation}
for $z \in H_t$, where $\gl \colon [0,T] \to \bbR$ is a continuous function called
the \emph{driving function} of the slit. Moreover, $\gl(t) = g_t(\gr(t))$ is the image of
the tip under the conformal map. 

\begin{figure}[t]
  \centerline{\includegraphics[scale=0.6]{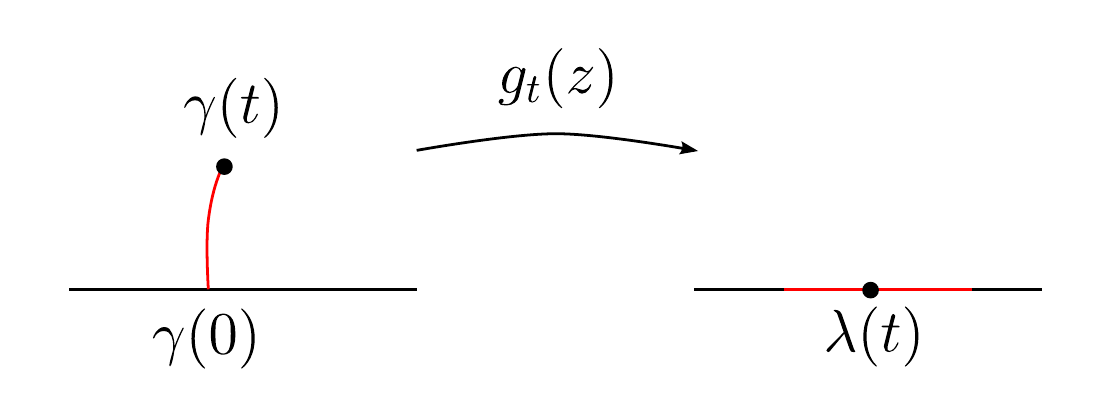}}
  \caption{A slit $\gr(t)$ and its driving function $\gl(t)$ are related by a conformal map.}
  \label{Fig:chordalLE}
\end{figure}

The foregoing procedure can be reversed. Suppose we are
given some continuous function $\gl \colon [0,T] \to \bbR$. For each $t \in [0,T]$, let $H_t$
be the set of points $z \in \bbH$ for which the solution of (\ref{LDE:forward_chordal}) is well-defined
up to time $t$, i.e.\ $g(s,z) \not= \gl(s)$ for $s \in [0,t]$. One can show that $H_t$ is a simply
connected region and $z \mapsto g(t,z)$ maps $H_t$ conformally onto $\bbH$ and satisfies the
hydrodynamic normalization. The set
$K_t := \bbH \setminus H_t$ is in general not a slit. Kufarev \cite{Kuf} constructed an example, in the
classical (radial) setting, for which a continuous
driving function does not generate a slit. The example can also be found in \cite[\S 3.4]{Dur}. Even
if the driving function is in $\Lip(\half)$, also known as $\half$-Holder continuous, the set $K_t$ may
not be locally connected (see \cite{MR} for an example).

Throughout this paper, we assume $\gl \colon [0,T] \to \bbR$ is $\Lip(\half)$, i.e.
\[
    \LipHalfNorm := \sup_{t_1 \not= t_2 \in [0,T]} \frac{\left| \gl(t_1) - \gl(t_1) \right|}
    {\left| t_1 - t_2 \right|^{\half}} < \infty.
\]
In 2005, Marshall and Rohde \cite{MR} showed\footnote{The theorems in \cite{MR} were proved in the radial
case. In a private communication, Don Marshall translated (with rigorous proof) the results to the chordal case.}
that there is an absolute constant $c_0 > 0$ so that for any $\gl \colon [0,T] \to \bbR$ with $\LipHalfNorm < c_0$,
the Loewner equation (\ref{LDE:forward_chordal}) generates a quasi-slit\footnote{By definition, a \emph{quasi-slit}
is a slit satisfying the Ahlfors three-point condition, i.e.\ there is a constant $L \geq 1$ such that for all points
$z_1$, $z_2$, $z_3$ on the slit in that order, $\left| z_1 - z_2 \right| + \left| z_2 - z_3 \right| \leq
L \left| z_1 - z_3 \right|$.} in the upper half-plane $\bbH$ and the slit meets $\bbR$ non-tangentially.
Lind \cite{Lin} proved that all these statements hold
for $c_0 = 4$, and this constant is the largest possible. On the other hand, in an unpublished paper
\cite{RTZ} Steffen Rohde, Huy Vo Tran and  Michel Zinsmeister give a sufficient condition for the driving function to
generate a rectifiable curve.

What more can we say if $\gl \colon [0,T] \to \bbR$ is more regular (smooth) than
$\Lip(\half)$? In \cite[page~59]{Ale} , a Russian book published in 1976, it was proved that
if $\gl \colon [0,T] \to \bbR$ has bounded first derivative then its slit is $C^1$. (The original statement was a radial
version. Here we formulate it in the chordal setting.) As of the writing of this paper
and up to the author's knowledge, it is the only result in the literature concerning the smoothness of a slit generated by
a driving function more regular than $\Lip(\half)$. Marshall and Rohde \cite{MR} implicitly suggest the following.

\begin{main}[heuristic version]
  $\gl \in C^{\gb} \Rightarrow \gr \in C^{\gb+\half}$ for $\gb > \half$.
\end{main}

In this paper, we will prove this statement for $\half < \gb \leq 2$, except when $\gb = \frac{3}{2}$ the slit $\gr$
is only proved to be weakly $C^{1,1}$. The precise statements are in Theorem~\ref{Thm:HolderDrivingFunction},
Theorem~\ref{Thm:C^(1,delta)} and Theorem~\ref{Thm:C^(1,1/2+delta)}, corresponding to the cases $\gb \in (\half, 1]$,
$\gb \in (1, \frac{3}{2}]$ and $\gb \in (\frac{3}{2},2]$. One of the key ingredients of our method is the Lipschitz continuity
(Theorem~\ref{Thm:mod_cont} below) of the map $\gl \mapsto \gr^{\gl}$, which was only known to be continuous \cite[Theorem~4.1]{LMR}.
Another ingredient is an integral representation of $\gr'(t)$, see Corollary~\ref{Cor:r'}.

\begin{refThm}{\ref{Thm:mod_cont}}[Lipschitz continuity]
  Suppose $\gl_1$, $\gl_2 \colon [0,T] \to \bbR$ satisfy $\Vert \gl_j \Vert_{\Lip(\half)} \leq 1$ for $j = 1$, 2. Then $\Vert
  \gr^{\gl_1} - \gr^{\gl_2} \Vert_{\infty} \leq c \, \Vert \gl_1 - \gl_2 \Vert_{\infty}$, where $c > 0$ is an absolute constant.
\end{refThm}

There is another natural and interesting question which we won't discuss in this paper but we mention it for the sake of
completion. If we know a slit $\gr$ is $C^n$, how smooth is its driving function? Earle and Epstein \cite{EE} answered this
question in 2001 for the radial case. Suppose $0 \in \gO \subseteq \bbC$ is a simply connected region and $\gr \colon (0,T] \to \gO$
is a slit avoiding the origin with base point $\gr(0) \in \partial \gO$. Let $R(t)$ be the conformal radius of $\gO_t := \gO
\setminus \gr((0,T])$ with respect to the origin. Earle and Epstein showed that if $\gr$ is $C^n$ regular on $(0,T]$ for some
integer $n \geq 2$, then the radial capacity $a(t) := - \log R(t)$ is $C^{n-1}$ on (0,T]. Moreover, if the slit is reparametrized
so that $a(t) = a(0) + t$, then its driving function $\gl$ is $C^{n-1}$ on $(a(0), a(T)]$. See \cite{EE} for the precise definitions
and statements. In the same paper, it was also proved that real analytic slits generate real analytic driving functions.

\begin{acknowledgement}
  The author would like to thank Steffen Rohde for numerous inspiring discussions and suggestions. The author offers heartfelt
  thanks to Vo Huy Tran, Don Marshall, James Gill, Joan Lind, Matthew Badger and Christopher McMurdie for their suggestions
  and comments in the earlier versions of this paper.
\end{acknowledgement}

\section{Definitions, notations and preliminaries} \label{S:definition} \label{S:definition}

\begin{gnotation}
  \begin{defList}
    \item[ ]  {\samepage
     \item  $a(\gve) \lesssim b(\gve)$} means $a(\gve) \leq C b(\gve)$ for some constant $C > 0$ (independent of $\gve$).
     \item  $a(\gve) \asymp b(\gve)$  means $a(\gve) \lesssim b(\gve)$ and $b(\gve) \lesssim a(\gve)$.
     \item  $a(\gve) \simeq b(\gve)$ as $\gve \to 0$ if $\frac{a(\gve)}{b(\gve)}$ has a positive and finite limit.
     \item  In this paper,  the lowercase $c$ is reserved to denote an absolute constant which may vary even in
               a single chain of equalities.
  \end{defList}                                                  
\end{gnotation}

\begin{definition}
  A \emph{slit} in $\bbH$ is a simple curve $\gr \colon [0,T] \to \overline{\bbH}$ with $\gr(0) \in \bbR$ and
  $\gr(t) \in \bbH$ for $0 < t \leq T$.
\end{definition}

All driving functions $\gl \colon [0,T] \to \bbR$ in this paper satisfy
\[
    \LipHalfNorm :=  \sup_{t_1 \not= t_2 \in [0,T]} \frac{\left| \gl(t_1) - \gl(t_1) \right|}
    {\left| t_1 - t_2 \right|^{\half}} < 4
\]    
(at least locally) and therefore generate slits by \cite{MR} and \cite{Lin}. We will use the following notations frequently.

\begin{notation}
  \begin{defList}
    \item[ ] {\samepage
    \item } The $\Lip(\half)$-norm \footnote{Strictly speaking, $\LipHalfNorm$ is only a semi-norm.} of $\gl \colon [0,T]
                \to \bbR$ is denoted by
                \[
                    \Vert \gl \Vert_{\Lip(\half,[0,T])} := \sup_{t_1 \not= t_2 \in [0,T]} \frac{\left| \gl(t_1) - \gl(t_2) \right|}
                    {\left| t_1 - t_2 \right|^{\half}}.
                \]
                Usually, we write $\LipHalfNorm$ instead of  $\Vert \gl \Vert_{\Lip(\half,[0,T])}$.
    \item   For positive integer $n \in \bbN$ and $0 < \ga \leq 1$, the $C^{n,\ga}$-norm of $\gl \colon [0,T] \to \bbR$ is
                \[
                    \Vert \gl \Vert_{C^{n,\ga}} =
                    \Vert \gl \Vert_{C^{n,\ga}([0,T])} := \sum_{k=0}^n \sup_{t \in [0,T]} \left| \gl^{(k)}(t) \right| +
                     \sup_{t_1 \not= t_2 \in [0,T]} \frac{\left| \gl^{(n)}(t_1) - \gl^{(n)}(t_2) \right|}{\left| t_1 - t_2 \right|^{\ga}}.
                \]
                For a slit $\gr$, its $C^{n,\ga}$-norm $\Vert \gr \Vert_{C^{n,\ga}}$ is defined similarly. If $\gb >1$ is not
                an integer, the notation $C^{\gb}$ refers to $C^{\left[\gb\right], \gb-\left[\gb\right]}$, where $\left[\gb\right]$
                is the integer part of $\gb$. For example, $C^{\frac{3}{2}}$ is the same as $C^{1,\half}$.
     \item  $\gr^{\gl} \colon [0,T] \to \overline{\bbH}$ denotes the slit generated by $\gl \colon [0,T] \to \bbR$.  When no
               confusion can occur, we write $\gr$ instead of $\gr^{\gl}$ for the sake of notation. The base of
               $\gr$ is $\gr(0) = \gl(0) \in \bbR$.
     \item  For each $t \in [0,T]$, $g_t \colon H_t \to \bbH$ denotes the (unique) conformal map from
                $H_t = \bbH \setminus \gr([0,t])$ onto the upper half-plane $\bbH$ satisfying  the normalization
                \[
                    g_t(z) = z + \frac{a_1(t)}{z} + \cdots
                \]
                as $z \to \infty$. All slits in this paper are parametrized by half-plane capacity, i.e.\ $a_1(t) = 2t$.
                Alternative notations such as $g(t,z)$ or $g^{\gl}_t(z)$ may be used interchangeably.
     \item  $f_t \colon \bbH \to H_t$ is the inverse function of $g_t$, i.e.\ $g_t(f_t(z)) = z$ for all $z \in \bbH$. We
                sometimes write $f(t,z)$ or $f_t^{\gl}(z)$.
     \item\label{Item:tau}
                For $0 \leq s < t \leq T$, we define $\gr_s(t) := g_s(\gr(t)) - \gl(s)$. To be flexible it may also be written
                as $\gr(s,t)$ or $\gr^{\gl}_s(t)$. The normalized version of $\gr(s,t)$ is $\tau(s,t) := \frac{\gr(s,t)}{\sqrt{t-s}}$.
                Note that $\tau(s,t) = 2i$ if and only if $\gl$ is constant on $[s,t]$.
     \item  In \S\ref{S:Lip(1/2+delta)}, we introduce the notation
                \[
                    L(s) = L^{\gl}(s) = \int_0^s \left[ \half + \frac{2}{\tau(s-u,s)^2} \right] \frac{du}{u}
                \]
                and show that $\gr'(s) = \frac{i}{\sqrt{s}} \, e^{L(s)}$ under appropriate assumptions.
   \end{defList}
\end{notation}  

\begin{definition}
  \begin{defList}
  \item[ ] {\samepage 
    \item }  We say that $\gl \colon [0,T] \to \bbR$ satisfies the \emph{\sigmaCondition}
  if $\gs \geq 0$ and
  \[
      \left| \gl(t_1) - \gl(t_2) \right| \leq \gs \left| t_1 - t_2 \right|^{\half}
  \]
  for all $t_1$,~$t_2 \in [0,T]$.
    \item  We say that $\gl \colon [0,T] \to \bbR$ satisfies the \emph{\MTdeltaCondition} if $\LipHalfNorm \leq 1$ and
               \[
                   \left| \gl(t_1) - \gl(t_2) \right| \leq M \left| t_1 - t_2 \right|^{\half + \gd}
               \]
               for all $t_1$,~$t_2 \in [0,T]$, where $M$, $T$,~$\gd > 0$.
    \item  We say that $\gl \colon [0,T] \to \bbR$ satisfies the \emph{\MTnAlphaCondition} if
               $\LipHalfNorm \leq 1$, $\gl \in C^{n,\ga}$ on $[0,T]$ and $\Vert \gl \Vert_{C^{n,\ga}} \leq M$,
               where $n \in \bbN$, $0 < \ga \leq 1$ and $M > 0$.
  \end{defList}
\end{definition}

We will not consider $\Lip(\half + \gd)$ driving functions until \S\ref{S:Lip(1/2+delta)}.
As a remark on terminologies, careful readers may see that
in (i) we do not make explicit reference to $T$ but we do in (ii). This is because
$\Lip(\half + \gd)$-norm is not invariant under Brownian scaling. The terminologies
reflect that all quantitative estimates in \S\ref{S:Lip(1/2+delta)} depend only on
$M$, $T$ and $\gd$, while in \S\ref{S:Lip(1/2)} our estimates are mostly in terms
of $\gs$.

In this paper, we use the diagram in Figure~\ref{Fig:LE_relation} to represent a
situation that $\gr(t)$ and $\gl(t)$ are related by the Loewner equation.

\begin{figure}[H]
  \centerline{\includegraphics[scale=0.6]{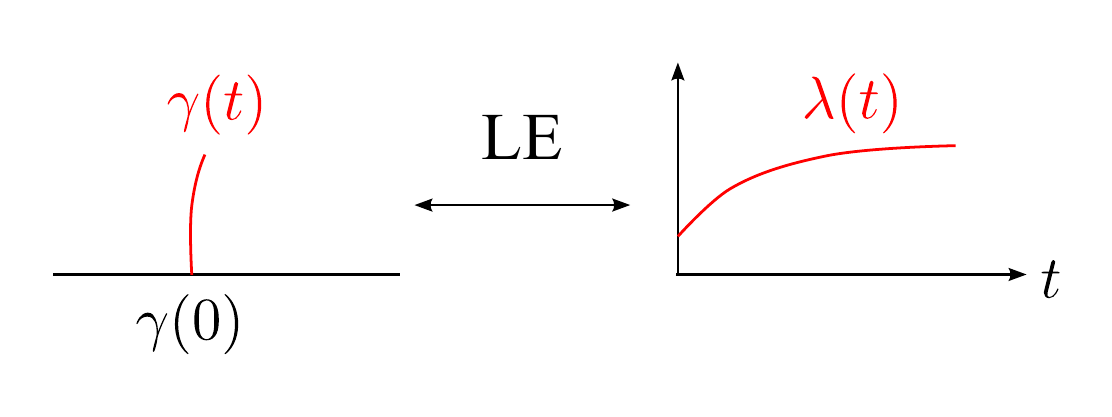}}
  \caption{$\gl(t)$ is the driving function of $\gr(t)$.}
  \label{Fig:LE_relation}
\end{figure}

For any continuous driving function $\gl \colon [0,T] \to \bbR$, the solution $g_t(z)$ of
(\ref{LDE:forward_chordal}) satisfies
\begin{align}
        \log g_t'(z) &= -  \int_0^t \frac{2}{[g_u(z) - \gl(u)]^2} \, du \label{E:logg'} \\
        g_t''(z) &= 4 g_t'(z) \int_0^t \frac{g_u'(z)}{[g_u(z) - \gl(u)]^3} \, du \label{E:g''}
\end{align}
for all $z \in H_t$. Equality (\ref{E:logg'}) can be derived easily if we differentiate (\ref{LDE:forward_chordal})
with respect to $z$, which gives
\[
    \frac{\partial}{\partial t} g_t'(z) = - \frac{2 g_t'(z)}{[g_t(z) - \gl(t)]^2}.
\]
To prove (\ref{E:g''}), we differentiate (\ref{E:logg'}) with respect to $z$. We comment that (\ref{E:logg'}) can
be used to estimate the size of $\left| g_s'(\gr(s+\gve)) \right|$ as $\gve \downarrow 0$, and this kind of estimates
is crucial in our work as well as other SLE problems. Equality (\ref{E:g''}) will be useful if one wants to obtain second
derivative estimates near the tip.

Equalities (\ref{E:logg'}) and (\ref{E:g''}) hold for any continuous driving function $\gl \colon [0,T] \to \bbR$. So
far we haven't made any smoothness assumption on $\gl$. We are going to do it in the coming sections.

\section{When $\gl \in \Lip(\half)$} \label{S:Lip(1/2)}

We begin by stating some useful facts.

\begin{fact} Suppose $\gl \colon [0,T] \to \bbR$ satisfies $\LipHalfNorm < 4$. \vspace{-0.2in}
  \begin{thmList}
    \item   (Scaling property)
                If we define $\glt \colon [0,1] \to \bbR$ by $\glt(s) := \frac{1}{\sqrt{T}} \left[ \gl(s T) - \gl(0) \right]$,
                then $\LipHalfNormt = \LipHalfNorm$ and for all $s \in [0,1]$,
                \[
                  \gr^{\glt}(s) = \frac{1}{\sqrt{T}} \, \left[ \gr^{\gl}(s T) -  \gr^{\gl}(0) \right].
                \]
                For example, suppose a slit $\gr$ is parametrized by half-plane capacity and $\gl$ is the driving function of $\gr$.
                The half-plane capacity reparametrization of the slit $3\gr(t)$ is $\widetilde{\gr}(t) = 3 \gr(\frac{t}{9})$. The
                scaling property says that the driving function of $\widetilde{\gr}$ is $\widetilde{\gl}(t) = 3 \gl(\frac{t}{9})$.

    \item  (Stationary property) For any $s \in (0,T)$, the \emph{time shift} $\gl_s \colon [0,T-s] \to \bbR$
               of $\gl$ is the function $\gl_s(u) := \gl(s+u)$. The corresponding slit is
               \[
                   \gr^{\gl_s}(u) = g_s(\gr(s+u)).
               \]
               See Figure~\ref{Fig:stationary} for an illustration.          
    \item  For any $t \in [0,T]$,
               \begin{equation} \label{E:r}
                   \gr^{\gl}(t) = \lim_{y \downarrow 0} f_t(\gl(t) + iy)
                                    = \gl(t) - \int_0^t \frac{2}{\gr(t-u,t)} \, du.
               \end{equation}
  \end{thmList}
\end{fact} 

\begin{figure}[ht]
  \centerline{\includegraphics[scale=0.6]{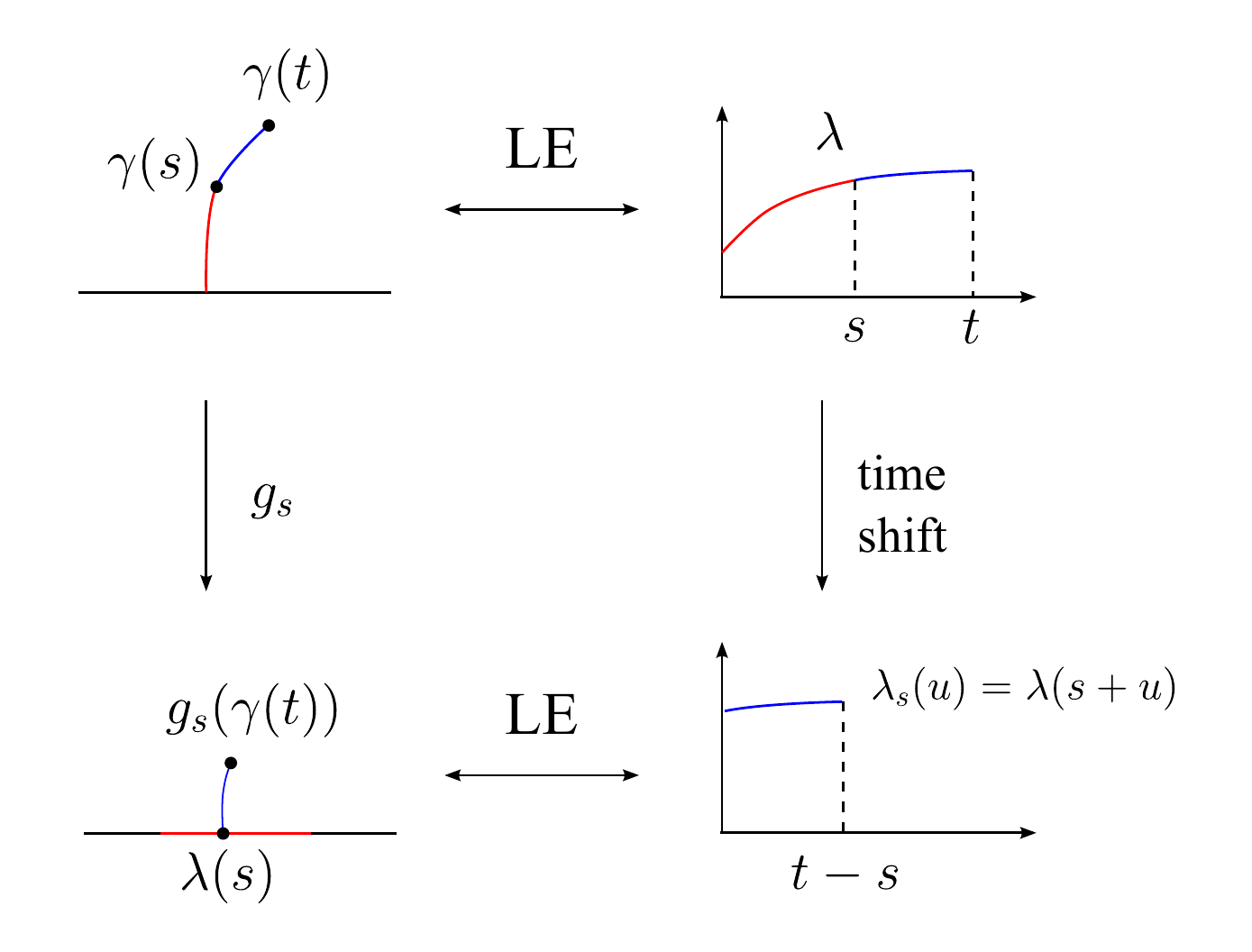}}
  \caption{The stationary property states that the above diagram commutes.} \label{Fig:stationary}
\end{figure}

The scaling property is extremely useful; in many situations it suffices to work only on the case $T = 1$.
The proofs of the scaling property and stationary property are elementary exercises. As we know from
\cite{MR} that $\gr([0,t])$ is a slit (in particular, locally connected), it follows from Caratheodory continuity
theorem (see, for example, \cite[Theorem~2.1]{Pom}) that the conformal map $f_t \colon \bbH \to
H_t$ is continuous at the boundary point $\gl(t)$. This proves the first equality in (\ref{E:r}). The second
equality is an immediate consequence of the Loewner differential equation (\ref{LDE:forward_chordal})
and the fundamental theorem of calculus applied to the function $u \mapsto g_u(\gr(t))$ ($0 \leq u \leq t$).
The first equality in (\ref{E:r}) is a non-trivial result for SLE curves,
whose driving functions are random and almost surely not $\Lip(\half)$ (see \cite{RS}, \cite{LSW04}).

For $0 \leq \gs < 4$, let $X_{\gs}$ be the space of all (continuous) functions
$\gl \colon [0,1] \to \bbR$ satisfying $\gl(0) = 0$ and $\LipHalfNorm \leq \gs$.
Under the supremum norm $\Vert \cdot \Vert_{\infty}$, the metric space $X_{\gs}$ is compact.
It is known \cite[Theorem~4.1]{LMR} that the map $X_{\gs} \to \bbH$ defined by $\gl \mapsto \gr^{\gl}(1)$ is
continuous. It follows that $E_{\gs} = \{ \gr^{\gl}(1) \colon \gl \in X_{\gs} \}$ is a compact subset
of $\bbH$.

By the scaling property,
\[
    \gr^{\gl}(t) \in \sqrt{t} \, E_{\gs} 
\]
for any $0 \leq t \leq 1$, $\gl \in X_{\gs}$, $0 \leq \gs < 4$. On the other hand,
it is easy to show (using compactness argument) that $E_{\gs}$ shrinks to
a singleton $\{ 2i \}$ as $\gs \to 0$. Our first question is: at what rate does
the diameter of $E_{\gs}$ go to zero?

\begin{lemma} \label{L:diamE}
  Suppose $\gl \colon [0,1] \to \bbR$ satisfies the \sigmaCondition\ with
  $0 \leq \gs < 4$ and $\gl(0) = 0$. Then
  \[
      \left| \Re \gr^{\gl}(1) \right| \leq \gs \qquad \text{and} \qquad
      4 - \gs^2 \leq [\Im \gr^{\gl}(1)]^2 \leq 4.
  \]
  In particular, $\diam(E_{\gs}) \leq c \, \gs$ for all $0 \leq \gs < 4$, where
  $c > 0$ is an absolute constant. 
\end{lemma}

\begin{figure}[H]
  \centerline{\includegraphics[scale=0.8]{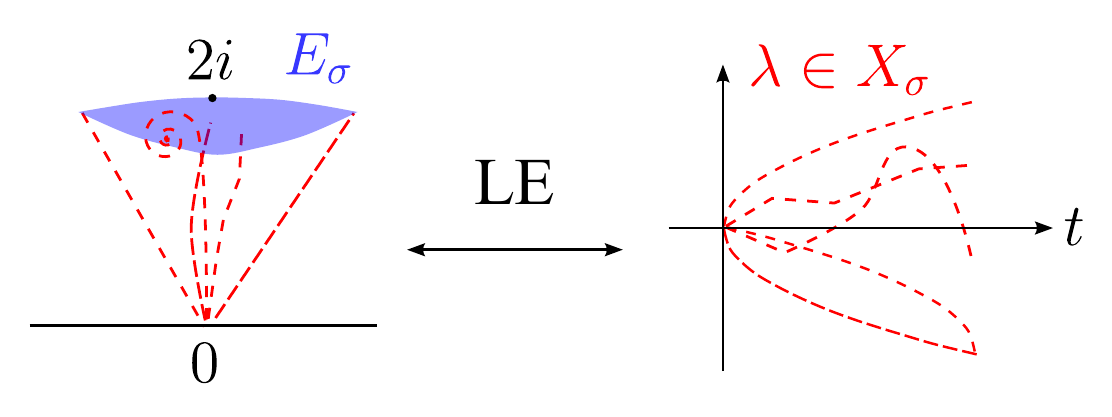}}
  \caption{A sketch of the compact set $E_{\gs}$ for $\gs = 1$.
                  When $\gs$ is close to zero, the set $E_{\gs}$ becomes very thin by Lemma~\ref{L:diamE}.}
                  \label{Fig:diamE}
\end{figure}

\begin{proof}
  Write $\gr(t) = \gr^{\gl}(t)$ for the sake of notation.
  The estimate $\left| \Re \gr(1) \right| \leq \gs$ follows from the simple observation
  that if $a \leq \gl(t) \leq b$ for all $t \in [0,1]$, $z_0 \in \bbH$ and $\Re(z_0) > b$
  (respectively $\Re(z_0) < a$), then $g_t(z_0)$ is defined and satisfies
  $\Re g_t(z_0) > b$ (respectively $\Re g_t(z_0) < a$) for all $t \in [0,1]$.

  To estimate $\Im \gr(1)$, we use the fact that
  \[
        \gr(1) = \lim_{y \downarrow 0} h_1(\gl(1) + iy),
  \]
  where $h_t(z)$ is the solution to the initial value problem
  \begin{equation} \label{IVP:reverseFlow}
      \left\{ \begin{aligned}
        \dot{h}_t(z) &= - \frac{2}{h_t(z) - \xi(t)} \\
        h_0(z) &= z
      \end{aligned} \right.
  \end{equation}
  with $\xi(t) := \gl(1-t)$. Fix any $y > 0$ and write $h_t(\xi(0) + iy)
  = x_t + iy_t$ ($x_t$, $y_t \in \bbR$). Let $A_t = (x_t - \xi(t))^2$ and
  $B_t = y_t^2$. By scaling, or by our argument in the beginning of the proof, $A_t \leq \gs^2 t$ for all
  $0 \leq t \leq 1$. Comparing the imaginary parts of (\ref{IVP:reverseFlow}), we have
  \[
      \dot{B}_t = \frac{4B_t}{A_t + B_t}.
  \]
  The obvious upper bound is $\dot{B}_t \leq 4$ and therefore $B_1 \leq 4 + y$, showing that
  $\Im h_t(\gl(1)+iy) \leq \sqrt{4 + y}$. Letting $y \downarrow 0$ gives $\Im \gr(1) \leq 2$.
  
  For the lower bound of $B_1$, we assume without loss of generality that $a := 4 - \gs^2 > 0$.
  (Otherwise, the lower bound is trivial.) Suppose for the sake of contradiction
  that $B_t < at$ for some $t \in [0,1)$. Let $T = \inf \{ t \geq 0 \colon B_t < at \}$. We
  have $0 < T < 1$, $B_T = aT$ and, for $0 \leq t \leq T$,
  \[
      \dot{B}_t = \frac{4B_t}{A_t + B_t} \geq \frac{4at}{\gs^2t + at} = a.
  \]
  This shows that $B_T \geq B_0 + aT > aT$, which is a contradiction. We have proved that
  $B_1 \geq 4 - \gs^2$.
\end{proof}
 
Consider the example
$\gl(t) = 2 \sqrt{\gk} \, (1 - \sqrt{1-t})$. When $0 < \gk < 4$, this driving function satisfies the
\sigmaCondition\ for $\gs = 2 \sqrt{\gk}$ and generates a logarithmic spiral with tip
\[
    \gr(1) = \sqrt{\gk} + i \sqrt{4-\gk} = \frac{\gs}{2} + i \sqrt{4 - \frac{\gs^2}{4}}.
\]    
(See \cite{KNK} for the computation and \cite{LMR} for a more conceptual approach.) This example
together with Lemma~\ref{L:diamE} show
\[
    c_1 \gs \leq \diam(E_{\gs}) \leq c_2 \gs.
\]
for all $0 < \gs < 4$, where $c_1$, $c_2 > 0$ are absolute constants.

The compactness of $E_{\gs}$ has a simple geometric consequence. If $\gl(0) = 0$ and $\gs = \LipHalfNorm
< 4$, by scaling we see that $\gr(t) \in \sqrt{t} E_{\gs}$ and the slit $\gr$ is contained in a cone whose
angle depends on $\gs$. If $\gl \in \Lip(\half + \gd)$, one has $\gr(t) \in \sqrt{t} E_{\gs(t)}$ with
$\gs(t) \lesssim t^{\gd}$ as $t \to 0$. The slit grows vertically. See Figure~\ref{Fig:vertically}.

\begin{figure}[H]
  \centerline{\includegraphics[scale=0.8]{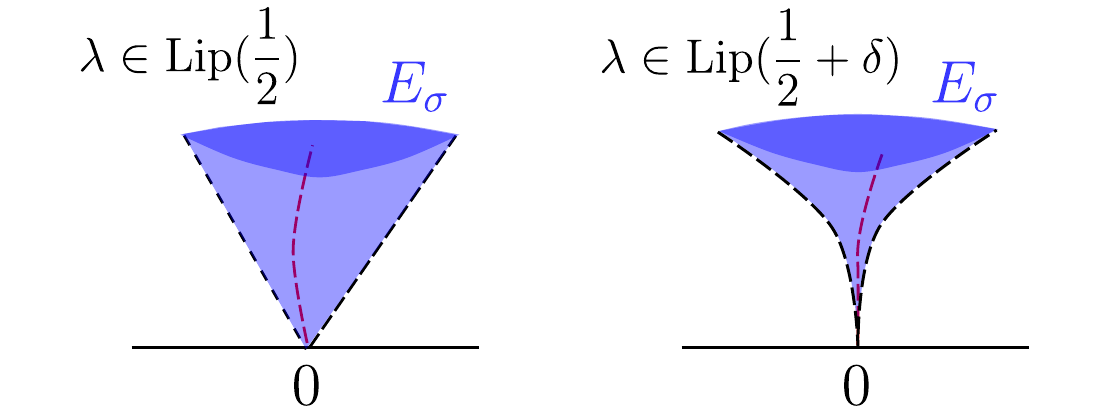}}
  \caption{One main difference between slits of $\Lip(\half)$ and $\Lip(\half + \gd)$ driving functions.}
  \label{Fig:vertically}
\end{figure} 

In \S\ref{S:Lip(1/2+delta)}, we will show that
\[
    \gr'(t) = \lim_{\gve \downarrow 0} \frac{i}{\sqrt{\gve} g_{t-\gve}'(\gr(t))}
    = \lim_{\gve \downarrow 0} \frac{i}{\sqrt{\gve} g_{t}'(\gr(t+\gve))}
\]
under appropriate smoothness assumption on $\gl$, and showing that
$\left| g_{t}'(\gr(t+\gve)) \right| \simeq \gve^{-\half}$ as $\gve \downarrow 0$ is more or less
equivalent to showing that $\gr'(t)$ exists. Of course, in this section we are still in the $\Lip(\half)$
case and do not expect $\gr(t)$ to be differentiable. The next lemma says that
$\left| g_{t-\gve}'(\gr(t)) \right| \asymp \gve^{-\half + O(\gs)}$, with an error term in
the exponent.

\begin{lemma} \label{L:f'_Lip(1/2)}
  If $\gl \colon [0,T] \to \bbR$ satisfies the \sigmaCondition\ for some $\gs \in [0,1]$, then
  for any $0 < s < t \leq T$,
  \[
      \left( \frac{t}{t-s} \right)^{\half - c \gs} \leq \left| g_s'(\gr(t)) \right| \leq
      \left( \frac{t}{t-s} \right)^{\half + c \gs},
  \]
  where $c > 0$ is an absolute constant.
\end{lemma}

\begin{proof}
  Without loss of generality, we assume $s = 1$ and $\gl(0) = 0$. Let $w = \gr(t)$. Then (\ref{E:logg'})
  gives
  \[
      \log g'_1(w) = \int_0^1 - \frac{2}{\tau(u,t)^2} \, \frac{du}{t-u},
  \]
  where $\tau(u,t) = \frac{g_u(w) - \gl(u)}{\sqrt{t-u}}$ was defined in Notation~(\ref{Item:tau}) in \S\ref{S:definition}.
  If the driving function $\gl$ is identically zero, $\tau^{\gl}(u,t)$ becomes $\tau^0(u,t) \equiv 2i$
  and the above equality reduces to
  \[
      \half \log \frac{t}{t-1} = \int_0^1 \half \, \frac{du}{t-u}.
  \]
  Subtracting the two equalities gives   
  \begin{equation} \label{E:log(g')_1}
      \log g_1'(w) - \half \log \frac{t}{t-1} = - \int_0^1 \left[ \half + \frac{2}{\tau(u,t)^2}
      \right] \frac{du}{t-u}.
  \end{equation}   
  By Lemma~\ref{L:diamE}, $\left| \half + \frac{2}{\tau(u,t)^2} \right| \leq c \gs$,
  where $c > 0$ is an absolute constant. Here we have implicitly used the condition $\gs \leq 1$,
  which guarantees that $\tau(u,t)$ stays in a fixed compact set $E_1 \subseteq \bbH$. The absolute
  constant $c$ in our last estimate is related to the derivative bound of the map $z \mapsto \frac{2}{z^2}$ on
  the compact set $E_1$.
  
  Finally, equation (\ref{E:log(g')_1}) gives
  \[
      \left| \log g'_1(w) - \half \log \frac{t}{t-1} \right|
      \leq \int_0^1 \frac{c \, \gs}{t-u} \, du = c \, \gs \log \frac{t}{t-1}.
  \]
\end{proof}

Lemma~\ref{L:diamE} and the following Theorem~\ref{Thm:mod_cont} will serve as two
fundamental tools for the rest of this paper.

\begin{theorem}[Lipschitz continuity] \label{Thm:mod_cont}
  Suppose $\gl$,~$\widetilde{\gl} \colon [0,T] \to \bbR$
  satisfy the \sigmaCondition\ for $\gs = 1$.
  Then, $\left\Vert \gr^{\gl} - \gr^{\widetilde{\gl}} \right\Vert_{\infty} \leq c \,
  \left\Vert \gl - \widetilde{\gl} \right\Vert_{\infty}$, where $c > 0$ is an absolute constant.
\end{theorem}

Fix any $T > 0$. Let $\widetilde{X}_{\gs}$ be the space of all (continuous) $\gl \colon [0,T] \to \bbR$
with $\LipHalfNorm \leq \gs$. Recently, Joan Lind, Don Marshall and Steffen Rohde proved \cite[Theorem~4.1]{LMR}
that the map $\gl \mapsto \gr^{\gl}$ is a continuous map from $(\widetilde{X}_{\gs}, \Vert \cdot \Vert_{\infty})$
into $(C([0,T]),\Vert \cdot \Vert_{\infty})$ for every $0 \leq \gs < 4$. Their proof uses the theory of quasi-conformal
maps. When $\gs \leq 1$, Theorem~\ref{Thm:mod_cont} says the map is Lipschitz continuous. 

For $\gs = 1$, the slit $\gr^{\gl}$ of $\gl \in X_{\gs}$ is contained in the cone $V = \{ z \in \bbH \colon \frac{\pi}{4} <
\arg(z) < \frac{3 \pi}{4} \}$. Theorem~\ref{Thm:mod_cont} remains true (with a larger
absolute constant $c$) if the constant 1 is replaced by a slightly larger number where
the slit is still contained in $V$. We do not know whether Theorem~\ref{Thm:mod_cont} holds
for $\gs = 4 - \gve$ when $\gve > 0$ is small.

\begin{proof}
  By scaling we can assume $T = 1$. Let $\gve := \sup_{0 \leq t \leq 1} \left| \gl(t) - \widetilde{\gl}(t) \right|$. It also
  loses no generality to assume $\gl(1) = \widetilde{\gl}(1)$. (If not, translate one of the slits by
  $\gl(1) - \widetilde{\gl}(1)$, which has absolute value at most $\gve$.) We extend $\gl$ so that $\gl(t) = \gl(1)$
  for all $t \geq 1$. Fix any small $\gd > 0$. The tip $\gr^{\gl}(1+\gd)$ is equal to $h_1$, where $h_t \colon [0,1] \to \bbC$
  is the solution of the backward Loewner differential equation
  \[
      \left\{
      \begin{aligned}
        \partial_t h_t &= - \frac{2}{h_t - \xi(t)} \\
        h_0 &= \xi(0) + 2i \sqrt{\gd}
      \end{aligned}
      \right.    
  \]
  and $\xi(t) := \gl(1-t)$. Similarly, we extend $\widetilde{\gl}$, define $\widetilde{\xi}(t)$, $\widetilde{h}_t$
  and let $Y(t) = h_t - \widetilde{h}_t$. Note that $Y(1) = \gr^{\gl}(1+\gd) - \gr^{\widetilde{\gl}}(1+\gd)$ and
  $Y(0) = 0$ since $\gl(1) = \widetilde{\gl}(1)$. By direct computation,
  \begin{equation} \label{DE:dh_t}
      \partial_t Y(t) = A(t) \left[ Y(t) + \left(\widetilde{\xi}(t) - \xi(t)\right) \right],
  \end{equation}
  where $A(t) = 2 (h_t - \xi(t))^{-1} (\widetilde{h}_t - \widetilde{\xi}(t))^{-1}$. We view (\ref{DE:dh_t}) as
  a first order linear ODE in $Y(t)$ and solve it using the method of integrating factor. Let $\mu(t) =
  \exp \left( - \int_0^t A(s) \, ds \right)$. One has $\frac{d}{dt} \left[ \mu(t) Y(t) \right] = \mu(t) A(t)
  \left(\widetilde{\xi}(t) - \xi(t)\right)$ and
  \[
      Y(1) = \int_0^1 \frac{\mu(s)}{\mu(1)} A(s) \left(\widetilde{\xi}(s) - \xi(s)\right) \, ds.
  \]
  We know $\left|\widetilde{\xi}(s) - \xi(s)\right| \leq \gve$ for all $s \in [0,1]$. To complete the proof,
  it remains to estimate the size of the integrating factor $\mu(t)$.    
  
  Notice that $A(t) \in \frac{1}{t+\gd} K$ for all $t \in [0,1]$, where
  \[
      K = \left\{ \frac{2}{z_1 z_2} \colon z_1, z_2 \in E_1 \right\}
  \]
  and $E_1$ is the compact set defined right before Lemma~\ref{L:diamE} (see Figure~\ref{Fig:diamE}).
  For convenience of the readers, we recall the definition
  \[
      E_1 := \left\{ \gr^{\gl}(1) \colon \gl(0) = 0 \text{ and } \Vert \gl \Vert_{\Lip(\half,[0,1])} \leq 1 \right\}.
  \]  
  By Lemma~\ref{L:diamE}, $K$ is contained in the left half-plane $\{ z \in \bbC \colon
  \Re(z) < 0 \}$. Let
  \[
      \gb = \inf \{ - \Re(z) \colon z \in K \} > 0.
  \]
  Since $A(t) \in \frac{1}{t+\gd} K$, we have $- \Re A(t) \geq \frac{\gb}{t + \gd}$ for all $t \in [0,1]$. For any
  $s \in [0,1]$,
  \[
      \frac{\mu(s)}{\mu(1)} = \exp \left[ \int_s^1 A(u) \, du \right] \qquad \text{and} \qquad
      \left| \frac{\mu(s)}{\mu(1)} \right| \leq (s + \gd)^{\gb}.
  \]
  Finally,
  \[
       \left| \gr^{\gl}(1+\gd) - \gr^{\widetilde{\gl}}(1+\gd) \right| =
        \left| Y(1) \right| \leq \gve \int_0^1 \left| \frac{\mu(s)}{\mu(1)} \right| \cdot \left| A(s) \right| \, ds
        \leq c \, \gve \int_0^1 (s+\gd)^{-1+\gb}  \, ds.
  \]
  where $c = \sup_{z \in K} \left| z \right| < \infty$ is an absolute constant. The result
  follows by letting $\gd \to 0$.
\end{proof}

\section{When $\gl \in \Lip(\half+\gd)$ with $0 < \gd \leq \half$} \label{S:Lip(1/2+delta)}

In this section, $\gl \colon [0,T] \to \bbR$ satisfies the \MTdeltaCondition, i.e.\ $\LipHalfNorm
\leq 1$ and
\[
      \left| \gl(t_1) - \gl(t_2) \right| \leq M \left| t_1 - t_2 \right|^{\half + \gd}
\]
for all $t_1$,~$t_2 \in [0,T]$, where $M$, $T$,~$\gd > 0$. The extra smoothness
allows us to improve the exponent in Lemma~\ref{L:f'_Lip(1/2)}.

\begin{lemma} \label{L:f'_Lip(1/2+delta)}
  If $\gl \colon [0,T] \to \bbR$ satisfies the \MTdeltaCondition\ for some $0 < \gd \leq \half$,
  then for any $0 < s < t \leq T$,
  \begin{equation} \label{I:g'_Lip(1/2+delta)}
      \frac{1}{C} \leq \sqrt{\frac{t-s}{t}} \left| g_s'(\gr(t)) \right| \leq C,
  \end{equation}
  where $C = C(M, T, \gd) > 0$. Moreover, for all $s \in (0,T)$, the limit
  \[
      \lim_{\gve \downarrow 0} \sqrt{\gve} g_s'(\gr(s+\gve)) = \sqrt{s}
      \exp \left\{ - \int_0^s \left[ \frac{1}{2u} + \frac{2}{\gr(s-u,s)^2}
      \right] du \right\}
  \]
  exists and is nonzero.
\end{lemma}

\begin{proof}
  As in the proof of Lemma~\ref{L:f'_Lip(1/2)},
  \[
      \log g_s'(\gr(t)) - \half \log \frac{t}{t-s} = - \int_0^s \left[ \half + \frac{2}{\tau(u,t)^2}
      \right] \frac{du}{t-u}.
  \]   
  The \MTdeltaCondition\ implies $\Vert \gl \Vert_{\Lip(\half,[u,t])} \leq M (t-u)^{\gd}$.
  Lemma~\ref{L:diamE} gives an estimate of our integral kernel:
  \[
      \left| \half + \frac{2}{\tau(u,t)^2} \right| \leq c \left| \tau(u,t) - 2i \right| \leq c M (t-u)^{\gd}
  \]
  for some absolute constant $c > 0$. (Again, we have implicitly used the condition $\LipHalfNorm \leq 1$,
  which guarantees that $\tau^{\gl}(u,t)$ stays in a fixed compact set $E_1 \subseteq \bbH$.) We have proved
  \[
      \left| \log g_s'(\gr(t)) - \half \log \frac{t}{t-s} \right| \leq cM \int_0^s (t-u)^{\gd-1} \, du = \frac{cM}{\gd}
      \left( t^{\gd} - (t-s)^{\gd} \right) \leq \frac{c M s^{\gd}}{\gd}
  \]
  and (\ref{I:g'_Lip(1/2+delta)}) follows. Taking $t = s + \gve$ with $\gve > 0$ gives
  \[
      \log \frac{\sqrt{\gve} g_s'(\gr(s+\gve))}{\sqrt{s+\gve}} = - \int_0^s
      \left[ \frac{1}{2} + \frac{2}{\tau(u,s+\gve)^2} \right] \frac{du}{s+\gve-u}.
  \]
  The existence of $\lim_{\gve \downarrow 0} \sqrt{\gve} g_s'(\gr(s+\gve))$ follows from the
  Lebesgue dominated convergence theorem.
\end{proof}

\begin{lemma} \label{L:diamK}
  Let $\gl_1$,~$\gl_2 \colon [0,T]  \to \bbR$ satisfy  the \MTdeltaCondition\ for some
  $0 < \gd \leq \half$. Suppose $\gl_1 = \gl_2$ on
  $[0,s]$ for some $s \in (0,T)$. Then, for any $\gve \in (0, T-s]$,
  \[
      \left| \gr^{\gl_1}(s+\gve) - \gr^{\gl_2}(s+\gve) \right| \leq C \gve^{1+\gd},
  \]
  where $C = C(M,T,\gd,s) > 0$.
\end{lemma}

\begin{proof}
  With $M$, $T$, $\gd$ fixed, let $X$ be the space of all functions $\gl \colon [0,T] \to \bbR$
  satisfying the \MTdeltaCondition\ and $\gl = \gl_1$ on $[0,s]$. For each $\gve \in (0,T-s]$,
  consider the compact set
  \[
      K_{\gve} = \left\{ \gr^{\gl}(s+\gve) \in \bbH \colon \gl \in X \right\}.
  \]
  
  \begin{figure}[H]
    \centerline{\includegraphics[scale=0.6]{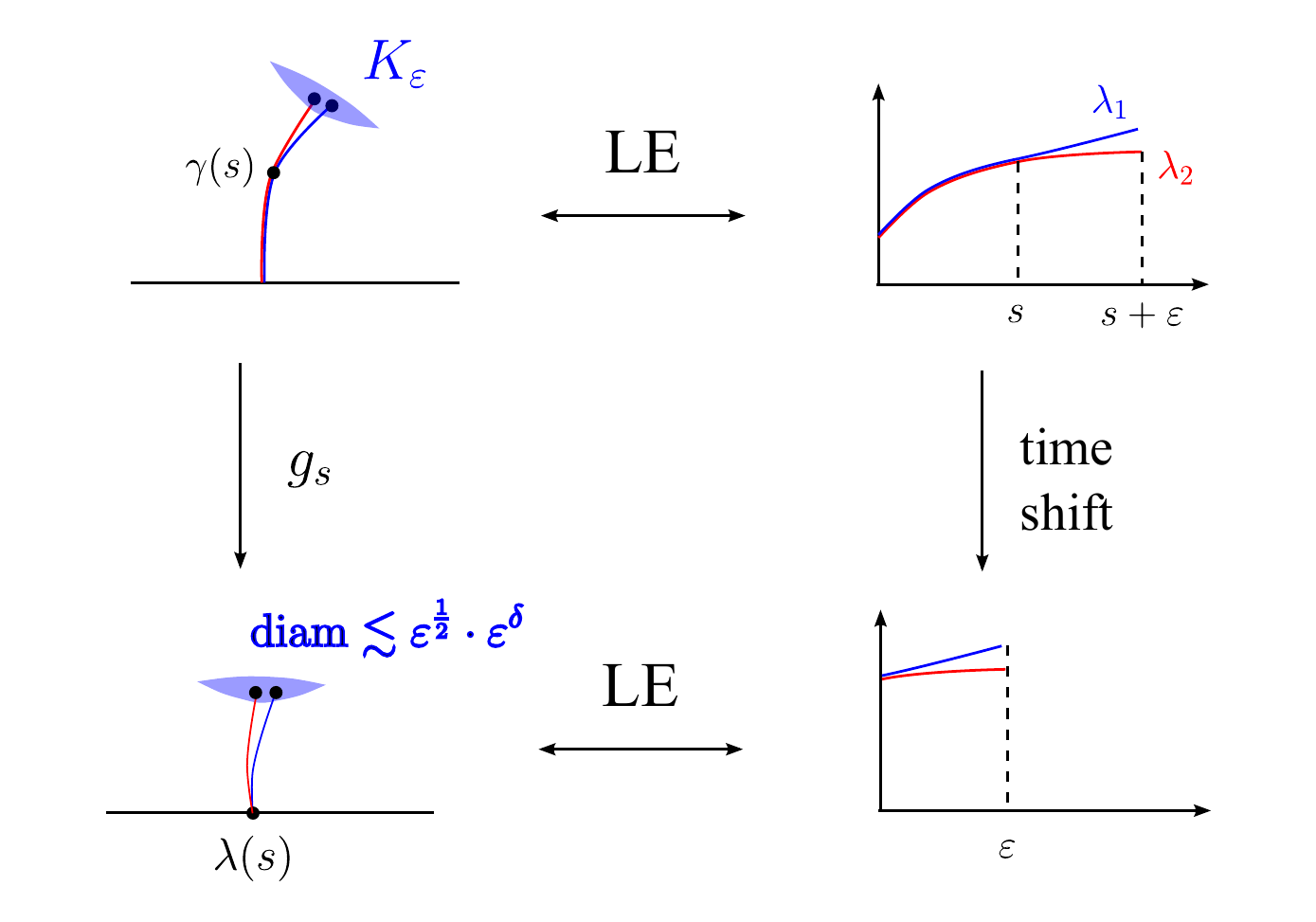}}
  \end{figure}
  
  It suffices to show $\diam(K_{\gve}) \leq C \gve^{1 + \gd}$.
  Let $g_s = g_s^{\gl_j}$ be the hydrodynamically normalized conformal map from
  $\bbH \setminus \gr^{\gl_j}([0,s])$ onto $\bbH$, and let $f_s = g_s^{-1}$. Note that
  \[
      \diam(K_{\gve}) \leq \diam(g_s(K_{\gve})) \cdot \sup_{z \in E} \left| f_s'(z) \right|,
  \]
  where $E$ is the convex hull of $g_s(K_{\gve})$. By Lemma~\ref{L:diamE},
  $\diam(g_s(K_{\gve})) \leq c M \gve^{\half + \gd}$. On the other hand,
  Lemma~\ref{L:f'_Lip(1/2+delta)} implies
  \begin{equation} \label{I:|f'|}
      \sup_{z \in g_s(K_{\gve})} \left| f_s'(z) \right| \leq C \sqrt{\gve}, 
  \end{equation}
  where $C = C(M,T,\gd,s) > 0$ does not depend on $\gve$. If we replace $g_s(K_{\gve})$ by its
  convex hull, the supremum in (\ref{I:|f'|}) can only increase by a bounded factor, by Koebe
  distortion theorem (see \cite{Pom}) and the fact that the hyperbolic diameter of $g_s({K_{\gve}})$
  is bounded above by some absolute constant $c_1$. Actually, we can take $c_1$ to be the
  hyperbolic diameter of the set $E_1$ in Figure~\ref{Fig:diamE}.
\end{proof}

\begin{corollary} \label{Cor:r'}
  Suppose $\gl \colon [0,T] \to \bbR$  satisfies the \MTdeltaCondition\ for some $0 < \gd \leq \half$,
  then $\gr = \gr^{\gl}$ is differentiable on $(0,T)$ and
  \begin{equation} \label{E:r'}
      \gr'(s) = \frac{i}{\sqrt{s}} \exp \left\{ \int_0^s \left[ \frac{1}{2u} + \frac{2}{\gr(s-u,s)^2} \right] du \right\}
  \end{equation}
  for all $s \in (0,T)$. At $s = T$, the left derivative $\gr_{-}'(T)$ exists and is given by the same formula.
\end{corollary}

\begin{proof}
  It suffices to show that the right derivative $\gr_{+}'(s)$ exists for $s \in (0,T)$ and is given by (\ref{E:r'}).
  (Using Theorem~\ref{Thm:mod_cont}, it is not hard to see that (\ref{E:r'}) is continuous in $s$, and it is
  an exercise to show that any right differentiable function on an open interval with continuous right derivative
  is in fact differentiable. A proof of this elementary fact can be found in \cite[Lemma~4.3]{Law}.)
  
  Fix any $s \in (0,T)$. We may assume without loss of generality that $\gl(t) = \gl(s)$
  for all $t \in [s,T]$, because modifying $\gl$ this way does not change the right derivative $\gr_{+}'(s)$, by
  Lemma~\ref{L:diamK}. We have
  $\gr(s+\gve) = f_s(\gl(s) + 2i\sqrt{\gve})$ and therefore
  \[
      \gr(s+\gve) - \gr(s) = \int_0^{\gve} \frac{i f_s'(\gl(s) + 2i\sqrt{u})}{\sqrt{u}} \, du
  \]
  for all $\gve \in (0,T-s]$. By Lemma~\ref{L:f'_Lip(1/2+delta)}, the integrand is continuous at $u = 0$. It follows
  that $\gr_{+}'(s)$ exists and is given by
  \[
      \gr_{+}'(s) = \lim_{\gve \downarrow 0} \frac{i f_s'(\gl(s) + 2i \sqrt{\gve})}{\sqrt{\gve}} =
      \frac{i}{\sqrt{s}} \exp \left\{ \int_0^s \left[ \frac{1}{2u} + \frac{2}{\gr(s-u,s)^2} \right] du \right\}.
  \]
\end{proof}

By formula (\ref{E:r'}), proving the smoothness of $\gr$ is equivalent to proving the smoothness of the integral,
which we call $L(s)$ from now on.

\begin{notation}
  Let
  \[
      L(s) = L^{\gl}(s) = \int_0^s \left[ \frac{1}{2u} + \frac{2}{\gr(s-u,s)^2} \right] du.
  \]
  This integral makes sense provided that $\gl \colon [0,T] \to \bbR$ satisfies the \MTdeltaCondition\ for some
  $0 < \gd \leq \half$. In the coming sections, when we impose more
  regularity assumptions on $\gl$, we will keep this notation.
\end{notation}

In the proof of Lemma~\ref{L:f'_Lip(1/2+delta)}, we have implicitly proved an upper bound of $\left| L(s) \right|$.
By (\ref{E:r'}), controlling the size of $\left| L(s) \right|$ gives an upper bound of $\left| \gr'(s) - \frac{i}{\sqrt{s}} \right|$.
This estimate will be useful later and we now explicitly state it.

\begin{lemma} \label{L:|L(s)|}
  Under the \MTdeltaCondition\ with $0 < \gd \leq \half$,
  \[
      \left| L(s) \right| \leq \frac{cMs^{\gd}}{\gd}
  \]
  for all $s \in [0,T]$, where $c > 0$ is an absolute constant. 
\end{lemma}

We will show that $L \in \Lip(\gd)$. For any $s \in (0,T)$ and $\gve \in [0,T-s]$,
\begin{equation} \label{E:deltaL}
      \begin{aligned}
        L(s+\gve) - L(s)  
        = &\int_0^s \left[ \frac{2}{\tau(s+\gve-u,s+\gve)^2} - \frac{2}{\tau(s-u,s)^2} \right] \frac{du}{u} +\\
        &\int_s^{s+\gve} \left[ \half + \frac{2}{\tau(s+\gve-u,s+\gve)^2}  \right] \frac{du}{u}. \\
      \end{aligned}
\end{equation}
Since $\Vert \gl \Vert_{\Lip(\half,[a,b])} \leq M(b-a)^{\gd}$, Lemma~\ref{L:diamE} gives $\left| \tau(s+\gve-u,s+\gve)
- 2i \right| \leq c M u^{\gd}$ and therefore
\begin{equation} \label{I:deltaL}
    \int_s^{s+\gve} \left| \half + \frac{2}{\tau(s+\gve-u,s+\gve)^2}  \right| \, \frac{du}{u} \leq \int_s^{s+\gve} c M u^{\gd-1}
    \, du \leq \frac{cM}{\gd} \gve^{\gd}.
\end{equation} 
The second integral in (\ref{E:deltaL}) is under control. The first integral can be estimated using the quantity
\begin{equation} \label{Def:w}
      \go(s,u,\gve) := \sup_{0 \leq v \leq u} \left| \gl(s+\gve-v) - \gl(s+\gve - u) - \gl(s-v) + \gl(s-u) \right|.
\end{equation}
Note that $\go(s,u,\gve)$ can be expressed as $\Vert \gl_1 - \gl_2 \Vert_{\infty}$, where
$\gl_1$, $\gl_2 \colon [0,u] \to \bbR$ are driving functions whose tips are $\gr(s+\gve-u,s+\gve)$ and $\gr(s-u,s)$.
Theorem~\ref{Thm:mod_cont} implies
\[
    \left| \gr(s+\gve-u,s+\gve) - \gr(s-u,s) \right| \leq c \, \go(s,u,\gve)
\]
for some absolute constant $c > 0$. This gives an estimate of the first integral in (\ref{E:deltaL}). We have proved
the following lemma.

\begin{lemma} \label{L:deltaL}
  If $\gl \colon [0,T] \to \bbR$ satisfies the \MTdeltaCondition\ for some $0 < \gd \leq \half$, then for any $0 \leq s <
  s+\gve \leq T$,
  \begin{equation} \label{I:deltaL}
    \begin{aligned}
      \left| L(s+\gve) - L(s) \right| &\leq c \int_0^s u^{-\frac{3}{2}} \, \go(s,u,\gve) \, du + \frac{cM}{\gd} \left[ (s+\gve)^{\gd}
                                                        - s^{\gd} \right] \\
                                                        &\leq c \int_0^s u^{-\frac{3}{2}} \, \go(s,u,\gve) \, du + \frac{cM\gve^{\gd}}{\gd},
    \end{aligned}  
  \end{equation}
  where $c > 0$ is an absolute constant $\go(s,u,\gve)$ is defined in (\ref{Def:w}).
\end{lemma}

The first inequality in (\ref{I:deltaL}) will be used in \S\ref{S:C^(1,delta)}, and we use the second estimate in this
section. When $\gd = \half$, we will see soon  (\ref{I:deltaL}) gives $\left| L(s+\gve) - L(s) \right| = O(\sqrt{\gve})$ as $\gve
\downarrow 0$. We achieve this by controlling the size of $\go(s,u,\gve)$. The estimate depends on the regularity
of $\gl$. In this section, $\gl$ is only $\Lip(\half + \gd)$, the following estimate of $\go(s,u,\gve)$ is what we should
expect and will be improved in \S\ref{S:C^(1,delta)} under the assumption $\gl \in C^{1,\gd}$.

\begin{lemma} \label{L:L_Lip(delta)}
  Let $\gl \colon [0,T] \to \bbR$ satisfy the \MTdeltaCondition.
  For any $0 \leq s < s+\gve \leq T$ and $0 < u < s$,
  \[
      \go(s,u,\gve) \leq
      \begin{cases}
        2M u^{\half + \gd}, & u \leq \gve \\
        2M \gve^{\half+\gd}, & u \geq \gve \\
      \end{cases}
  \]
  and $\left| L(s + \gve) - L(s) \right| \leq C \gve^{\gd}$,
  where $C = C(M,T, \gd) > 0$.  
\end{lemma}

\begin{proof}
  When $u \leq \gve$,
  \[
    \begin{aligned}
       &\left| \gl(s+\gve-v) - \gl(s+\gve - u) - \gl(s-v) + \gl(s-u) \right| \\
       \leq  &\left| \gl(s+\gve-v) - \gl(s+\gve - u)\right| + \left|\gl(s-v) - \gl(s-u) \right| \\
       \leq &2Mu^{\half + \gd}
    \end{aligned}    
  \]
  for any $0 \leq v \leq u$. When $u \geq \gve$, we rearrange terms as
  \[
    \begin{aligned}
       &\left| \gl(s+\gve-v) - \gl(s+\gve - u) - \gl(s-v) + \gl(s-u) \right| \\
       \leq  &\left| \gl(s+\gve-v) - \gl(s-v) \right| + \left| \gl(s+\gve - u) - \gl(s-u) \right| \\
       \leq &2M\gve^{\half + \gd}
    \end{aligned}    
  \]
  for any $0 \leq v \leq u$. We have proved the desired estimates of $\go(s,u,\gve)$.
  
  To prove $\left| L(s + \gve) - L(s) \right| \leq C \gve^{\gd}$, we split the integral in (\ref{I:deltaL}):
  \[
          \left| \int_0^s u^{-\frac{3}{2}} \, \go(s,u,\gve) \, du \right| \leq \int_0^{\gve} 2M u^{-1+\gd} \, du +
          \int_{\gve}^s 2M u^{-\frac{3}{2}} \gve^{\half+\gd} \, du \leq C \gve^{\gd}
  \]
\end{proof}

We have all the ingredients for proving our first main result.

\begin{theorem} \label{Thm:HolderDrivingFunction}
    Let $\gl \colon [0,T] \to \bbR$ be such that
    \[
              \left| \gl(t_1) - \gl(t_2) \right| \leq M \left| t_1 - t_2 \right|^{\half + \gd}
    \]
    for all $t_1$,~$t_2 \in [0,T]$, where $M > 0$ and $\gd \in (0, \half]$ are constants. Then
    \begin{thmList}
         \item  $\gR(t)  := \gr(t^2)$ is $C^{1,\gd}$ regular\footnote{whose derivative is nonzero everywhere} on $[0,\sqrt{T}]$; and
         \item  the slit $\gr(t)$ grows vertically at $t = 0$.
    \end{thmList}  
    With an extra assumption that $\LipHalfNorm \leq 1$, these statements are
     quantitative:
     \begin{equation} \label{I:Gamma_C1delta_norm}
          \Vert \gR \Vert_{C^{1,\gd}([0,T])} \leq N \qquad \text{and} \qquad
          \inf_{t \in [0,T]} \left| \gR'(t) \right| \geq \frac{1}{N},
     \end{equation}
     where $N = N(M,T,\gd) > 0$ depends only on $M$, $T$ and $\gd$. Furthermore,
      \begin{equation} \label{I:Gamma'}
            \left| \gR'(t) - 2i \right| \leq N \, t^{2 \gd} \qquad
            (0 < t \leq \sqrt{T}).
      \end{equation}
\end{theorem}

\begin{proof}[Proof of Theorem~\ref{Thm:HolderDrivingFunction}]
  We first assume $\LipHalfNorm \leq 1$.
  By Corollary~\ref{Cor:r'}, $\gR(t) = \gr(t^2)$ is differentiable and $\gR'(t) = 2 i e^{L(t^2)}$. With this formula
  of $\gR'(t)$, we claim that $\gR'(t)$ is $\Lip(\gd)$. To see this, we
  first note from Lemma~\ref{L:|L(s)|} that $\sup_{0 \leq t \leq \sqrt{T}}\left| L(t^2) \right| \leq R$ for some constant
  $R = \frac{cMT^{\gd}}{\gd}$ depending only on $M$, $T$ and $\gd$. This tells us
  \[
      \left| \gR'(t_1) - \gR'(t_2) \right| \leq 2 \left( \sup_{\left| z \right| \leq R} \left|e^z\right| \right) \left| L(t_1^2) -
      L(t_2^2) \right| \leq C \left| t_1 - t_2 \right|^{\gd}
  \]
  for some $C = C(M,T,\gd) > 0$ by Lemma~\ref{L:L_Lip(delta)}. This proves (a). The estimates
  (\ref{I:Gamma_C1delta_norm}) and (\ref{I:Gamma'}) follow from Lemma~\ref{L:|L(s)|} and other estimates
  we have proved. For example,
  \[
      \left| \gR'(t) \right| = 2 \left| e^{L(t^2)} \right| = 2 e^{\Re L(t^2)} \geq 2 e^{-R}
  \]
  and (\ref{I:Gamma'}) can be derived from
  \[
      \left| \gR'(t) - 2i \right| = 2 \left| e^{L(t^2)} - 1 \right| \leq 2 \left( \sup_{\left|z\right| \leq R} \left| e^z \right| \right)
      \left| L(t^2) \right|.
  \]
  
  If $\LipHalfNorm > 1$, we pick a partition $0 = t_0 < t_1 < \cdots < t_n = T$ for which $M(t_{j+2} - t_j)^{\gd} < 1$
  for all $j = 0$, 1, \dots,~$n-2$. This guarantees that $\Vert \gl \Vert_{\Lip(\half, [t_j,t_{j+2}])} < 1$ for all $j$.
  It suffices to show
  \begin{defList}
    \item  $\gR(t) = \gr(t^2)$ is $C^{1,\gd}$ regular on $[t_0,t_2]$; and
    \item  $\gr(t)$ is $C^{1,\gd}$ regular on $[t_j, t_{j+2}]$ for $j = 1$, 2, \dots,~$n-2$.
  \end{defList}
  (i) follows from the argument of the previous paragraph, since we are back to the case $\LipHalfNorm \leq 1$.
  To show (ii), we pick $\gve > 0$ for which $M(t_{j+2} - t_j + \gve)^{\gd} \leq 1$.  Again, by our earlier argument
  for the case $\LipHalfNorm \leq 1$, the function
  \[
      u \mapsto \gr(t_j - \gve, t_j + u)
  \]
  is $C^{1,\gd}$ regular on $[0, t_{t+2} - t_j]$, and therefore
  \[
      u \mapsto \gr(t_j + u) = f_{t_j - \gve}(\gl(t_j - \gve) +\gr(t_j - \gve, t_j + u))
  \]
  is also $C^{1,\gd}$ regular on $[0, t_{j+2} - t_j]$. 
\end{proof}

\section{When $\gl \in C^{1,\gd}$ with $0 < \gd \leq \half$} \label{S:C^(1,delta)}

In this section, our driving function $\gl \colon [0,T] \to \bbR$ satisfies the \MToneDeltaCondition\
with $0 < \gd \leq \half$. That is to say, $\LipHalfNorm \leq 1$, $\gl \in C^1([0,T])$ and
\[
    \Vert \gl \Vert_{C^{1,\gd}} = \sup_{t \in [0,T]} \left| \gl(t) \right| + \sup_{t \in [0,T]} \left| \gl'(t) \right|
    + \sup_{t_1 \not= t_2 \in [0,T]} \frac{\left| \gl'(t_1) - \gl'(t_2) \right|}{\left| t_1 - t_2 \right|^{\gd}} \leq M.
\]
Our goal is to improve the estimate $\left| L(s+\gve) - \L(s) \right| = O(\gve^{\gd})$ given in \S\ref{S:Lip(1/2+delta)}, 
and we are expecting $O(\gve^{\half + \gd})$. Any $C^{1,\gd}$ function is Lipschitz, i.e.\ $\Lip(\half + \half)$. Applying
Lemma~\ref{L:|L(s)|} and the first inequality in (\ref{I:deltaL}) with $\gd = \half$ yields
\begin{equation} \label{I:L}
    \left| L(s) \right| \leq cM \sqrt{s}
\end{equation}
and
\begin{equation} \label{I:deltaL_01}
    \left| L(s+\gve) - L(s) \right| \leq c \int_0^s u^{-\frac{3}{2}} \, \go(s,u,\gve) \, du + \frac{cM \gve}{\sqrt{s}},
\end{equation}
for $0 < s < s+\gve \leq T$. We now improve the estimate of $\go(s,u,\gve)$ in Lemma~\ref{L:L_Lip(delta)} to the following.
Recall the definition
\[
    \go(s,u,\gve) := \sup_{0 \leq v \leq u} \left| \gl(s+\gve-v) - \gl(s+\gve - u) - \gl(s-v) + \gl(s-u) \right|.
\]

\begin{lemma} \label{L:w(s,u,epsilon)}
  Let $\gl \colon [0,T] \to \bbR$ satisfy the \MTnAlphaCondition\ with $n = 1$ and  $0 < \ga \leq 1$. For any $0 \leq s < s+\gve \leq T$
  and $0 < u < s$,
  \[
      \go(s,u,\gve) \leq
      \begin{cases}
        M \gve^{\ga} u, & u \leq \gve \\
        M u^{\ga} \gve, & u \geq \gve \\
      \end{cases}
  \]
  If $\ga = \gd < \half$, for any $0 < s < s+\gve \leq T$, we have
  \begin{equation} \label{I:deltaL_02}
      \left| L(s+\gve) - L(s) \right| \leq \frac{cM}{1-2\gd} \cdot \left( \gve^{\half + \gd} + \frac{\gve}{\sqrt{s}} \right),
  \end{equation}
 where $c > 0$ is an absolute constant. When $\gd = \half$,
  \begin{equation} \label{I:deltaL_03}
      \left| L(s+\gve) - L(s) \right| \leq c M \gve \cdot \left[ 1 + \log^{+} \left(\frac{s}{\gve}\right) + \frac{1}{\sqrt{s}} \right],
  \end{equation}
  where $\log^{+} x = \max\{ \log(x), 0 \}$.
\end{lemma}

\begin{proof}
  The equalities
  \[
      \begin{aligned}
          &\gl(s+\gve-v) - \gl(s+\gve-u) - \gl(s-v) + \gl(s-u) \\
          = &\int_{s-u}^{s-v} \gl'(w+\gve) - \gl'(w) \, dw \\
          = &\int_{s-u}^{s+\gve-u} \gl'(w+u-v) - \gl'(w) \, dw
      \end{aligned}
  \]
  hold for all $v \in [0,u]$. Since $\gl' \in \Lip(\ga)$, they prove the desired estimate of $\go(s,u,\gve)$. If $\ga = \gd \in (0,\half)$,
  (\ref{I:deltaL_01}) and our estimates of $\go(s,u,\gve)$ give the following. (We assume $s > \gve$ in the
  computation below. When $s \leq \gve$, the integral $\int_{\gve}^s$ will disappear and our estimate still holds.)   
  \[
      \begin{aligned}
          \left| L(s+\gve) - L(s) \right| &\leq c M \int_0^{\gve} u^{-\frac{3}{2}} \gve^{\gd} u \, du + cM \int_{\gve}^{s}
          u^{-\frac{3}{2}} u^{\gd} \gve \, du + \frac{cM\gve}{\sqrt{s}} \\
          &\leq c M \gve^{\half + \gd} + \frac{cM \gve^{\half + \gd}}{\half - \gd} + \frac{cM\gve}{\sqrt{s}}
      \end{aligned}
  \]
  This proves (\ref{I:deltaL_02}). When $\gd = \half$, the second term in the last expression should be replaced by
  $cM\gve \log^{+} \frac{s}{\gve}$, which shows (\ref{I:deltaL_03}).
\end{proof}

Combining Theorem~\ref{Thm:HolderDrivingFunction} and Lemma~\ref{L:w(s,u,epsilon)}, we have the following theorem.

\begin{theorem} \label{Thm:C^(1,delta)}
  Suppose $\gl \colon [0,T] \to \bbR$ satisfies the \MToneDeltaCondition\ with $0 < \gd < \half$.
  Then  the curve $\gR(t) := \gr(t^2)$ is $C^{1, \half + \gd}$ regular
  on $[0,\sqrt{T}]$. In fact,
  \[
      \Vert \gR \Vert_{C^{1, \half + \gd}([0,T])} \leq N \qquad \text{and} \qquad \inf_{t \in [0,T]} \left| \gR'(t) \right| \geq \frac{1}{N},
  \]
  where $N = N(M,T,\gd) > 0$. When $\gd = \half$, $\gR'(t)$ is weakly Lipschitz in the sense that
  \begin{equation} \label{I:GammaWeaklyLip}
      \left| \gR'(t_1) - \gR'(t_2) \right| \leq N \left| t_1 - t_2 \right| \max \left\{ 1, \log \frac{1}{\left|t_1-t_2\right|}  \right\} \qquad
      (0 \leq t_1 < t_2 \leq \sqrt{T}),
  \end{equation}
  where $N = N(M,T) > 0$.
\end{theorem}

We do not know whether or not $\limsup_{\gd \uparrow \half} N(M,T,\gd) < \infty$. The (non-optimal) constant $C = \frac{cM}
{1-2\gd}$ in inequality (\ref{I:deltaL_02}) blows up when $\gd \uparrow \half$. When $\gd = \half$, inequality (\ref{I:deltaL_03})
implies $\left| \gR'(s+\gve) - \gR'(s) \right| = O(\gve \log^{+} \frac{1}{\gve})$ as $\gve \downarrow 0$. In particular, $\gR(t)$ is weakly
$C^{1,1}$ in the sense that it is $C^{1,\ga}$ for every $\ga < 1$.

\begin{proof}
  We know from Theorem~\ref{Thm:HolderDrivingFunction} that $\gR \in C^1$. Since $\gR'(s) = 2i
  e^{L(s^2)}$, all we need to show is that $s \mapsto L(s^2)$ is $\Lip(\half + \gd)$. Suppose $0 < s < s+\gve \leq \sqrt{T}$.
  If $s \leq \gve$, then (\ref{I:L}) gives
  \[
      \left| L((s+\gve)^2) - L(s^2) \right| \leq \left| L((s+\gve)^2) \right| + \left| L(s^2) \right| \leq C \gve
  \]
  for some constant $C > 0$. If $s > \gve$, Lemma~\ref{L:w(s,u,epsilon)} implies 
  \[
      \begin{aligned}
          \left| L((s + \gve)^2) - L(s^2) \right| &\leq C (2s\gve + \gve^2)^{\half + \gd} + \frac{C (2s \gve + \gve^2)}{s} \\
                                                                       &\leq C(2s+\gve)^{\half+\gd} \gve^{\half+\gd} + 3C\gve \\
                                                                       &\leq C \gve^{\half + \gd}.
      \end{aligned}                                                                 
  \]
  Suppose $\gd = \half$. To prove (\ref{I:GammaWeaklyLip}), it suffices to show
  \[
      \left| L((s+\gve)^2) - L(s^2) \right| \leq C \gve \left( 1 + \left| \log \gve \right| \right)
  \]
  for $0 \leq s < s+\gve \leq \sqrt{T}$ and some constant $C > 0$. As before, when $s \leq 2\gve$ the desired estimate
  follows from (\ref{I:L}). When $s > 2\gve$, (\ref{I:deltaL_03}) implies
  \[
      \begin{aligned}
        \left| L((s+\gve)^2) - L(s^2) \right| &\leq C (2s+\gve) \gve \left( 1 + \log^+
        \frac{s^2}{(2s+\gve) \gve} + \frac{1}{s} \right) \\
        &\leq C s \gve \left( 1 + \log \frac{s}{2\gve} + \frac{1}{s} \right) \\
        &\leq C \gve \left( 1 + \left| \log \gve \right| \right)
      \end{aligned}  
  \]
  where $C = C(M,T) > 0$.
\end{proof}

\begin{corollary}
  Under the assumptions of Theorem~\ref{Thm:C^(1,delta)}, the slit $\gr(t) = \gr^{\gl}(t)$ is $C^{1, \half + \gd}$ regular on
  $[a,T]$ (or weakly $C^{1,1}$ when $\gd = \half$) for every $a > 0$. When $0 < \gd < \half$,
  \[
      \Vert \gr \Vert_{C^{1, \half + \gd}([a,T])} \leq N \qquad \text{and} \qquad \inf_{t \in [a,T]} \left| \gr'(t) \right| \geq \frac{1}{N},
  \]
  where $N = N(M,T,\gd,a) > 0$. When $\gd = \half$, the statement is also quantitative.
\end{corollary}

\section{When $\gl \in C^{1, \half + \gd}$ with $0 < \gd \leq \half$} \label{S:C^(1,1/2+delta)}

In this section, $\gl \colon [0,T] \to \bbR$ satisfies the \MTnAlphaCondition\ for $n=1$ and $\ga = \half + \gd$, where
$0 < \gd \leq \half$. That is to say, $\LipHalfNorm \leq 1$ and
\[
    \Vert \gl \Vert_{C^{1,\half + \gd}} =
    \sup_{t \in [0,T]} \left| \gl(t) \right| + \sup_{t \in [0,T]} \left| \gl'(t) \right| +
    \sup_{t_1 \not= t_2 \in [0,T]} \frac{\left| \gl'(t_1) - \gl'(t_2) \right|}{\left| t_1 - t_2 \right|^{\half + \gd}} \leq M.
\]
Our goal is to show $\gr \in C^{2, \gd}$ on $[a,T]$ for every $a > 0$, which is equivalent to proving $L \in C^{1,\gd}$ on the
same interval.

Since $\gl \in C^{1,\half + \gd}$, it is in particular $C^{1,\half}$ and we know from Lemma~\ref{L:w(s,u,epsilon)} that $L^{\gl}(s)$
is weakly Lipschitz on $[a,T]$ for every $a > 0$. We claim that $L(s)$ is differentiable on $(0,T]$ and $L'(s) \in \Lip(\gd,[a,T])$.
By (\ref{E:deltaL}), at least formally one has
\begin{equation} \label{E:L'}
    \begin{aligned}
        L'(s) &= \frac{1}{2s} + \frac{2}{\gr(s)^2} + \int_0^s \partial_s \left[ \frac{2}{\gr(s-u,s)^2} \right] \, du \\
                &= \frac{1}{2s} + \frac{2}{\gr(s)^2} - 4 \int_0^s \frac{\partial_s \gr(s-u,s)}{\gr(s-u,s)^3} \, du.
    \end{aligned}            
\end{equation}
To see that this formula is valid, we must show that $\partial_s \left[ \frac{2}{\gr(s-u,s)^2} \right] $ is integrable over $u \in [0,s]$.

\begin{lemma} \label{L:C^(1,alpha)}
  Let $\gl \colon [0,T] \to \bbR$ satisfy the \MTnAlphaCondition\ with $n = 1$ and  $0 < \ga \leq 1$. For any $0 < u \leq s < s+\gve
  \leq T$, we have
  \begin{align}
        \left| \gr(s+\gve-u,s+\gve) - \gr(s-u,s) \right| &\leq c M \min \left( u \gve^{\ga}, \gve u^{\ga} \right) \label{I:delta_r}\\
        \left| \partial_s \gr(s-u,s) \right| &\leq cM u^{\ga} \label{I:|dr/ds|} \\
        \left| \partial_s \gr(s+\gve-u,s+\gve) - \partial_s \gr(s-u,s) \right| &\leq C \gve^{\ga} \label{I:delta_dr/ds},
  \end{align}  
  where $C = C(M,T) > 0$ and $c > 0$ is an absolute constant. When $\ga = \half + \gd$ with $0 < \gd \leq \half$, $L(s)$ is
  differentiable for $s \in (0,T]$ and $L'(s)$ is given by (\ref{E:L'}). Moreover, $L'(s) \in \Lip(\gd)$ on $[a,T]$ for every $a > 0$.
\end{lemma}

\begin{proof}
  By Theorem~\ref{Thm:mod_cont}, $\left| \gr(s+\gve-u,s+\gve) - \gr(s-u,s) \right| \leq c \, \go(s, u, \gve)$. Inequality
  (\ref{I:delta_r}) follows immediately from Lemma~\ref{L:w(s,u,epsilon)}, and it implies
  \[
      \left| \frac{\gr(s+\gve-u,s+\gve) - \gr(s-u,s)}{\gve} \right| \leq c M u^{\ga}.
  \]
  Letting $\gve \to 0$ gives (\ref{I:|dr/ds|}). To prove (\ref{I:delta_dr/ds}), we differentiate $\gr(s-u,s) = g_{s-u}(\gr(s)) - \gl(s-u)$:
  \[
      \begin{aligned}
          \partial_s \gr(s-u,s) &= \frac{2}{g_{s-u}(\gr(s)) - \gl(s-u)} + g'_{s-u}(\gr(s)) \, \gr'(s) - \gl'(s-u) \\
                                             &= \frac{2}{\gr(s-u,s)} + \gr'_{s-u}(s) - \gl'(s-u)
      \end{aligned}   
  \]
  The last term $\gl'(s-u)$ is $\Lip(\ga)$ in $s$ by assumption. The term $\frac{2}{\gr(s-u,s)}$ is also $\Lip(\ga)$ in $s$
  by (\ref{I:delta_r}):
  \[
      \begin{aligned}
          \left| \frac{2}{\gr(s+\gve-u,s+\gve)} - \frac{2}{\gr(s-u,s)} \right| &=
          \frac{2 \left| \gr(s+\gve-u,s+\gve) -  \gr(s-u,s) \right|}{\left| \gr(s+\gve-u,s+\gve) \right| \cdot \left| \gr(s-u,s) \right|} \\
          &\leq \frac{c M u \, \gve^{\ga}}{\sqrt{u} \cdot \sqrt{u}} \\ &= c M \gve^{\ga}.
      \end{aligned}
  \]  
  The remaining term $\gr'_{s-u}(s)$ is given by $\gr'_{s-u}(s) = \frac{i}{\sqrt{u}} \exp L(s-u,s)$, where
  \[
      L^{\gl}(s-u,s) = L(s-u,s) := \int_0^u \left[ \frac{1}{2v} + \frac{2}{\gr(s-v,s)^2} \right] dv. 
  \]
  Note that $L^{\gl}(s-u,s) = L^{\widetilde{\gl}}(u)$, where $\widetilde{\gl}(t) = \gl(s-u+t)$ is a time shift of $\gl$. From
  Lemma~\ref{L:|L(s)|} we know that $\left| L(s-u,s) \right| \leq c M \sqrt{u} \leq c M T$. On the ball $\{ z \in \bbC \colon
  \left| z \right| \leq cMT \}$ the function $e^z$ has bounded derivative, so
  \[
      \begin{aligned}
           \left| \gr'_{s+\gve-u}(s+\gve) - \gr'_{s-u}(s) \right| &= \frac{1}{\sqrt{u}}
           \left| e^{L(s+\gve-u,s+\gve)} - e^{L(s-u,s)} \right| \\
           &\leq \frac{C}{\sqrt{u}}
           \left| L(s+\gve-u,s+\gve) - L(s-u,s) \right|
      \end{aligned}
  \]
  where $C = e^{cMT}$. On the other hand,
  \[
      \begin{aligned}
          \left| L(s+\gve-u, s+\gve) - L(s-u,s) \right| &\leq \int_0^u \left| \frac{2}{\gr(s+\gve-v,s+\gve)^2}
          - \frac{2}{\gr(s-v,s)^2} \right| \, dv \\
          &\leq c \int_0^u v^{-\frac{3}{2}} \go(s, v, \gve) \, dv \\
          &\leq c \int_0^u v^{-\frac{3}{2}} M v \gve^{\ga} \, dv \\
          &\leq c M \sqrt{u} \, \gve^{\ga}
      \end{aligned}
  \]
  These imply $\left| \gr'_{s+\gve-u}(s+\gve) - \gr'_{s-u}(s) \right| \leq C \gve^{\ga}$ with $C = C(M,T) > 0$, and (\ref{I:delta_dr/ds})
  holds.
  
  If $\half < \ga \leq 1$, we will show $L(s)$ is differentiable on $(0,T]$ and $L'(s)$ is given by (\ref{E:L'}). By extending $\gl$, we
  can assume without loss of generality that $0 < s < T$. For small $\gve > 0$,
  \[
      \begin{aligned}
          \frac{L(s+\gve) - L(s)}{\gve} = &\frac{1}{\gve} \int_s^{s+\gve} \left[ \frac{1}{2u} + \frac{2}{\gr(s+\gve-u,s+\gve)^2} \right] du + \\
          &\int_0^s \frac{1}{\gve} \left[ \frac{2}{\gr(s+\gve-u,s+\gve)^2} - \frac{2}{\gr(s-u,s)^2} \right] du
      \end{aligned}    
  \]
  It is not hard to see that the integrand in the last term is dominated by $C u^{\ga - \frac{3}{2}}$, which is integrable
  since $\ga > \half$. By Lebesgue dominated convergence theorem, the second integral
  converges as $\gve \downarrow 0$. Convergence of the first integral follows from continuity and does not require $\ga > \half$. Since
  $L(s)$ has a continuous right derivative, it is differentiable on $(0,T)$.
\end{proof}

We now prove the main result in this section.

\begin{theorem} \label{Thm:C^(1,1/2+delta)}
  Suppose $\gl \colon [0,T] \to \bbR$ satisfies the \MTnAlphaCondition\ with $n=1$, $\ga = \half + \gd$ and $0 < \gd \leq \half$.
  Then the slit $\gr(t) = \gr^{\gl}(t)$ is $C^{2, \gd}$ regular on $[a,T]$ for every $a > 0$. The
  statement is quantitative in the sense that
  \[
      \Vert \gr \Vert_{C^{2, \gd}([a,T])} \leq N \qquad \text{and} \qquad \inf_{t \in [a,T]} \left| \gr'(t) \right| \geq \frac{1}{N},
  \]
  where $N = N(M,T,\gd,a) > 0$ depends only on $M$, $T$, $\gd$ and $a$.
\end{theorem}

\begin{proof}
  Any $C^{1,\half+\gd}$ driving function $\gl$ is in particular $C^{1,\half}$. Applying Theorem~\ref{Thm:C^(1,delta)},
  we know that $\gr$ is weakly $C^{1,1}$ regular on every $[a,T]$. All we need to show is that $\gr''$ exists and is $\Lip(\gd)$ on $[a,T]$.
  
  By Lemma~\ref{L:C^(1,alpha)}, $L$ is differentiable and therefore $\gr''$ exists on $(0,T]$ and is given by
  \begin{equation} \label{E:r''}
      \gr''(s) = \frac{2 \gr'(s)}{\gr(s)^2} - 4 \gr'(s) Q(s),
  \end{equation}
  where
  \[
      Q(s) := \int_0^s \frac{\partial_s \gr(s-u,s)}{\gr(s-u,s)^3} \, du.
  \]
  The first term $\frac{2 \gr'(s)}{\gr(s)^2}$ in (\ref{E:r''}) is $\Lip(\gd)$ on $[a,T]$ because both $\gr'$ and $\gr$ are
  $\Lip(\gd)$ and the size of the denominator $\left| \gr(s) \right|^2 \asymp s$ is bounded below by positive constant.   
  It remains to prove $Q \in \Lip(\gd,[0,T])$. The integral kernel of $Q$ has the form $K(x,y) = \frac{x}{y^3}$. We have
  \begin{equation} \label{I:delta_Q}
      \begin{aligned}
          \left| Q(s+\gve) - Q(s) \right| \leq &\int_0^s \left| K(x+\gD x, y + \gD y) - K(x,y) \right| \, du + \\
                                                                   &\int_s^{s+\gve} \left| K(x+\gD x, y+\gD y) \right| \, du,
      \end{aligned}                                                             
  \end{equation}
  where $x = \partial_s \gr(s-u,s)$, $y = \gr(s-u,s)$, $\gD x = \partial_s \gr(s+\gve-u,s+\gve) - \partial_s \gr(s-u,s)$ and
  $\gD y = \gr(s+\gve-u,s+\gve) - \gr(s-u,s)$. By Lemma~\ref{L:C^(1,alpha)} and scaling,
  \[
      \left| K(x+\gD x, y+\gD y) \right| \leq \frac{c M u^{\half + \gd}}{u^{\frac{3}{2}}} = c M u^{- 1 + \gd}
  \]
  and the last term in (\ref{I:delta_Q}) is bounded by
  \[
      \int_s^{s+\gve} \left| K(x+\gD x, y+\gD y) \right| \, du \leq \frac{cM}{\gd} \left[ (s+\gve)^{\gd} - s^{\gd} \right]
      \leq \frac{c M}{\gd} \gve^{\gd},
  \]
  whenever $0 \leq s < s+\gve \leq T$. On the other hand, we split the first integral in (\ref{I:delta_Q}) into two terms
  and handle them separately. If $s \leq \gve$, by triangle inequality and Lemma~\ref{L:C^(1,alpha)},
  \[
      \left| K(x+\gD x, y + \gD y) - K(x,y) \right| \leq \left| K(x+\gD x, y + \gD y) \right| + \left| K(x,y) \right| \leq
      c M u^{-1 + \gd}.
  \]
  Integrating gives
  \[
      \int_0^{\gve} \left| K(x+\gD x, y + \gD y) - K(x,y) \right| \, du \leq \frac{c M}{\gd} \gve^{\gd}.
  \]
  If $s \geq \gve$, we still need to estimate the integral from $\gve$ to $s$.
  \[
      \begin{aligned}
          &\left| K(x+\gD x, y + \gD y) - K(x, y) \right| \\
          \leq &\left| K(x+\gD x, y) - K(x, y) \right| + \left| K(x+\gD x, y + \gD y) - K(x+\gD x, y) \right| \\
          \leq &\left| \gD x \right| \cdot \sup \left| \partial_x K \right| + \left| \gD y \right| \cdot \sup \left| \partial_y K \right| \\
          \leq &C \gve^{\half + \gd} u ^{-\frac{3}{2}} + C \gve u^{-1 + 2\gd}
      \end{aligned}
  \]
  by Lemma~\ref{L:C^(1,alpha)} again. Integrating gives
  \[
      \int_{\gve}^s \left| K(x+\gD x, y + \gD y) - K(x,y) \right| \, du \leq C \gve^{\gd}
  \]
  with $C = C(M, T, \gd) > 0$.
\end{proof}

For $\gl(t) \in C^{1, \half + \gd}$ with $0 < \gd \leq \half$, Theorem~\ref{Thm:C^(1,1/2+delta)} shows that $\gr(t)$
is $C^{2,\gd}([a,T])$ for every $a > 0$. Certainly $\Vert \gr \Vert_{C^1([a,T])} \to \infty$ as $a \downarrow 0$.
To obtain smoothness up to $t = 0$, one has to reparametrize the slit, and a natural candidate is $\gR(t) = \gr(t^2)$.
We do not know whether $\gR(t)$ is $C^{2,\gd}$ up to $t = 0$, but assuming this we have a quadratic approximation
\begin{equation} \label{E:r(t^2)QuadApprox}
    \gR(t) = \gr(t^2) = 2it + \frac{2}{3} \gl'(0) t^2 + O(t^{2+\gd})
\end{equation}
as $t \to 0$. The heuristic reason is that on a small interval close to the origin any $C^{1,\ga}$ driving function $\gl(t)$
can be approximated by a fixed linear function $\gl'(0)t$. For driving functions of the form $\gl(t) = at$ ($a > 0$) the
quadratic approximation of $\gr(t^2)$ can be explicitly computed.

\begin{example}
  Let $\gl(t) = t$. Since a linear function is invariant under time shift, $\gr(s-u, s) = \gr(u)$ does
  not depend on $s$, and we have
  \[
      L(s) = \int_0^s \left[ \frac{1}{2u} + \frac{2}{\gr(u)^2} \right] du \qquad \text{and} \qquad L'(s) = \frac{1}{2s} + \frac{2}{\gr(s)^2}.
  \]
  For this case it is possible to explicitly compute the series expansion of $L(s)$ near $s = 0$. The reader may refer to
  \cite{KNK} for the computation. As $s \to 0$, one has $\gr(s) = 2i\sqrt{s} + \frac{2}{3}s + O(s^{\frac{3}{2}})$
  and $L'(s) = - \frac{i}{3 \sqrt{s}} + O(1)$. Note that $[L(s^2)]' = 2sL'(s^2) \to -\frac{2i}{3}$. The function $s \mapsto L(s^2)$
  is $C^1$ up to $s = 0$. Since $\gR'(s) = 2i \exp L(s^2)$, the curve $\gR(s) = \gr(s^2)$ is $C^2$ up to $s = 0$ and has
  a quadratic approximation $\gR(s) = 2is + \frac{2}{3} s^2 + O(s^3)$. Note that this agrees with (\ref{E:r(t^2)QuadApprox}).
\end{example}

For any constant $b > 0$, the driving function $\gl_b(t) = bt$ can be obtained from $\gl_1(t) = t$ and a Brownian
scaling: $\gl_b(t) = \frac{1}{b} \gl_1(b^2 t)$. The computation in the above example gives
\[
    \gr^{\gl_b}(s) = \frac{1}{b} \gr^{\gl_1}(b^2 s) = 2i \sqrt{s} + \frac{2b}{3} s + O(s^{\frac{3}{2}})
\]
as $s \to 0$. We have just verified (\ref{E:r(t^2)QuadApprox}) for all driving functions of the form $\gl_b(t) = bt$ with $b \in \bbR$.
(The case $b < 0$ follows from symmetry.)

\begin{proposition} \label{Prop:Gamma''(0)}
  Suppose $\gl \colon [0,T] \to \bbR$ satisfies the \MTnAlphaCondition\ with $n=1$, $\ga = \half + \gd$ and $0 < \gd \leq \half$.
  Then $\gR(t) = \gr(t^2)$ is twice differentiable everywhere on $[0,\sqrt{T}]$ and $\gR''(0) = \frac{4}{3} \gl'(0)$.
\end{proposition}

\begin{proof}
  We already know $\gR(t)$ is $C^2$ on $(0,\sqrt{T}]$ (Theorem~\ref{Thm:C^(1,1/2+delta)}) and still need to show the existence
  of $\gR''(0)$. By comparing $\gl(t)$ with the linear driving function $\widetilde{\gl}(t) = \gl'(0)t$, we will show that
  $s \mapsto L^{\gl}(s^2)$ is differentiable at $s = 0$. To simplify the notations, we write
  $\widetilde{L}(\cdot) = L^{\widetilde{\gl}}(\cdot)$ and $\widetilde{\tau}(\cdot, \cdot) = \tau^{\widetilde{\gl}}(\cdot, \cdot)$.
  Notice that
  \[
    \begin{aligned}
      \left| L(s^2) - \widetilde{L}(s^2) \right| &\leq \int_0^{s^2} \left| \frac{2}{\tau(s^2-u, s^2)^2} - \frac{2}{\widetilde{\tau}(s^2-u, s^2)^2} \right|
      \, \frac{du}{u} \\
      &\leq c \int_0^{s^2} \left| \tau(s^2-u,s^2) - \widetilde{\tau}(s^2-u,s^2) \right| \, \frac{du}{u}
    \end{aligned}  
  \]
  Using the condition $\gl \in C^{1,\half+\gd}$, we can estimate the $\Vert \cdot \Vert_{\infty}$ distance between the two
  driving functions which generate $\tau(s^2-u,s^2)$ and $\widetilde{\tau}(s^2-u,s^2)$. The Lipschitz continuity
  Theorem~\ref{Thm:mod_cont} will then imply
  \[
       \left| L(s^2) - \widetilde{L}(s^2) \right| = O(s^{1 + 2\ga}).
  \]
  For the purpose of computing $\lim_{s \to 0} \frac{L(s^2)}{s}$, we can replace $\gl(t)$ by $\gl'(0) t$ without
  affecting the existence of the limit and its value. Since we are able to compute this limit for linear driving functions,
  it follows that
  \[
      \left. \frac{d L(s^2)}{ds} \right\vert_{s=0}  =  \lim_{s \to 0} \frac{L(s^2)}{s} = \lim_{s \to 0} \frac{\widetilde{L}(s^2)}{s} = - \frac{2i}{3} \gl'(0).
  \]
  From the formula $\gR'(s) = 2i \exp L(s^2)$ and the above computation, we have $\gR''(0) = \frac{4}{3} \gl'(0)$.
\end{proof}

\end{document}